\documentclass[12pt]{amsart}
\usepackage{amsmath,amsthm,amsfonts,amssymb,amscd, bbold}
\usepackage{a4wide}
\usepackage{hyperref}

 \usepackage[pdftex]{graphicx}

\newcommand{\var}{\operatornamewithlimits{var}}

\renewcommand{\epsilon}{{\varepsilon}}

\newcommand{\cC}{{\mathcal C}}

\newcommand{\cI}{{\mathcal I}}

\newcommand{\NN}{{\mathbb N}}
\newcommand{\RR}{{\mathbb R}}

\theoremstyle{plain}
\newtheorem{lemma}{Lemma}[section]

\newtheorem{theorem}{Theorem}

\newtheorem*{theoremstar}{Theorem}

\theoremstyle{definition}

\newtheorem{remark}[lemma]{Remark}

\begin{document}

\title[Invariant measures for maps with critical and singular points]
{Invariant measures for interval maps \\ 
with critical points and singularities}

\author{V\'\i tor Ara\'ujo} 
\address{Instituto de Matematica, UFRJ, 
CP 68.530, 
Rio de Janeiro 21.945-970
Brazil and 
Centro de Matematica, Universidade do Porto, Rua do Campo
Alegre 687, 4169-007 Porto, Portugal}
\email{vitor.araujo@im.ufrj.br; vdaraujo@fc.up.pt}

\author{Stefano Luzzatto}
\address{Mathematics Department, Imperial College London, SW7 2AZ, UK}
\urladdr{http://www.ma.ic.ac.uk/~luzzatto}
\email{stefano.luzzatto@imperial.ac.uk}

\author{Marcelo Viana}
\address{Instituto de Matematica Pura e Aplicada, Est. D. Castorina
110, Rio de Janeiro, Brazil}
\urladdr{http://www.impa.br/~viana}
\email{viana@impa.br}

\thanks{This research was partly supported by a Royal Society
International Joint Project Grant and EPSRC grant number GR T09699
01. V.A. was also supported by FAPERJ, CNPq (Brazil) and CMUP-FCT
(Portugal). M.V. was also supported by CNPq (Brazil), 
FAPERJ, and PRONEX-Dynamical Systems. 
The authors acknowledge the hospitality of Imperial College
London and IMPA Rio de Janeiro where most of this work was carried
out. Thanks also to Colin Little and to the anonymous referee 
for their careful reading of a
preliminary version of the paper and for their useful remarks.}

\date{20 February 2009}


\begin{abstract}
We prove that, under a mild summability condition on the
growth of the derivative on critical orbits any piecewise monotone 
interval map possibly containing discontinuities and singularities
with infinite derivative (cusp map)
admits an ergodic
invariant probability measures which is absolutely continuous with
respect to Lebesgue measure.
\end{abstract}

\maketitle

\section{Introduction and statement of results}
\label{s.int}

\subsection{Introduction}
The existence of absolutely continuous invariant probability measures
(\emph{acip}'s) for dynamical systems
is a problem with a history going back more than 70 years, see for
example pioneering papers by Hopf \cite{Hop32} and Ulam and von
Neumann \cite{UlaNeu47}. Notwithstanding an extensive amount of
research in this direction in the last two or three decades, the
problem is still not completely solved even in the one-dimensional
setting which is the focus of this paper. 
Quite general conditions are known which guarantee the existence of
\emph{acip}'s for uniformly expanding maps in the smooth case or
possibly admitting singularities, i.e. discontinuities with possibly
unbounded derivatives (see \cite{Via97}\cite{Luz06} for additional
remarks and references), and for smooth maps with a finite number of
critical points (see \cite{BruLuzStr03} for first and  strongest results
including decay of correlations, and 
\cite{BruRivSheStr08} for the most recent and possibly the most
general conditions for the existence of absolutely continuous
invariant measures in this setting) and even for smooth maps with a
countable number of critical points \cite{AraPac07}. We are
interested here in a  general class of maps which contain  critical
points \emph{and} singularities. 

\begin{figure}[h]
\includegraphics{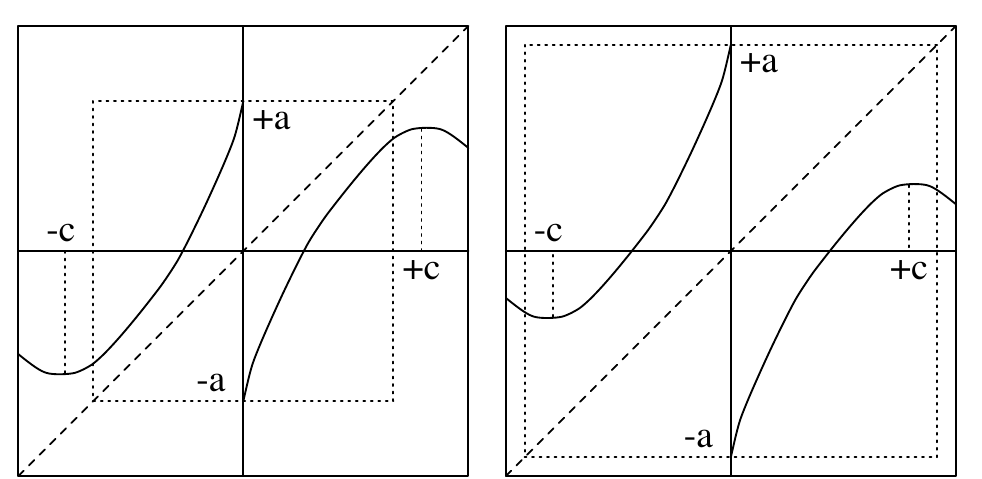}
\caption{Interval maps with critical points and singularities}
\end{figure}
A natural family of maps belonging to this class was introduced in
\cite{LuzTuc99, LuzVia00} and motivated by the study of the
return map of the Lorenz equations near classical parameter values,
see Figure 1. It is clear from the arguments in these papers, that
the presence of both critical points and singularities and their
interaction can give rise to significant technical as well as
fundamental issues. In particular, as we shall see in the present
setting, it is not enough to have just some expansivity conditions in
order to obtain the existence of an acip, as expansivity might occur
due to the regions of unbounded derivative even when the deeper
dynamical structure of the map is very pathological. Moreover, it is
possible that the interaction of critical points and singularities
could give rise to new phenomena which are still unexplored.

\subsubsection{Exponential growth and subexponential recurrence}

Some general results for the existence of \emph{acip}'s and their
properties in  maps with critical points and singularities were
obtained in \cite{AlvLuzPin04} under the assumption that
\emph{Lebesgue almost every point} satisfy some exponential
derivative growth and  subexponential recurrence conditions. These
conditions provide an interesting conceptual picture but may be hard
to verify in practice. 
On the other hand, it was proved in \cite{LuzTuc99} \cite{LuzVia00}
that with positive probability in the parameter space of Lorenz-like
families, 
the orbits of the \emph{critical points}
 satisfy such exponential derivative growth and  subexponential
recurrence conditions. In \cite{DiaHolLuz06} it was shown, within a
more general setting of maps with multiple critical points and
singularities, that these conditions are in fact sufficient to guarantee
the existence of an ergodic \emph{acip} (from which it can in fact be
proved that Lebesgue almost every point also satisfies such
conditions).

\subsubsection{Summability conditions}\label{sumcon}
Our aim in this paper is to obtain the same conclusion but relax as
much as possible the conditions on the orbit of the critical points,
to include in particular cases in which the derivative growth may be
subexponential and/or the recurrence of the critical points
exponential. A crucial observation  concerning the difference between
the smooth case and the case with singularities discussed here is
that in the smooth case, for which in particular the derivative is
bounded, any condition on the growth of the derivative
 is also implicitly a condition on the recurrence to the critical
set. Indeed sufficiently strong recurrence to the critical set will
always kill off any required derivative growth. On the other hand,
this is not the case in our setting. Derivative growth may be
exponential but arise as a consequence  of very strong recurrence to
the singularities even if we have at the same time very strong
recurrence to the critical set. Strong recurrence to either the
singular or the critical set brings its own deep structural problems
and can be an intrinsic obstruction to the existence of an
\emph{acip}.  We shall formulate below 
a condition which simultaneously
keeps track of the growth of the derivative along critical orbits and 
of the recurrence of such orbits to the critical set within a single
summability condition. This optimizes the result to include a larger
class of maps than would be possible by having to independent
conditions both of which need to be satisfied.  We
conjecture that it is not possible to obtain a general result on the
existence of acip's in the presence of both critical points and
singularities by assuming only conditions on the derivative growth of
critical points.

\subsection{Statement of results}

We now give the precise statement of our result. 
We let \( \mathcal F \) denote the class of interval map satisfying
the conditions formulated in 
Sections \ref{C1}, \ref{C2} and \ref{C3} below. Then we have the
following

\begin{theoremstar}\label{th:physical}
Every map  \( f \in \mathcal F \) admits a finite number of
absolutely continuous invariant (physical)  probability measures
whose basins cover \( I \) up to a set of measure 0. 
    \end{theoremstar}

\subsubsection{Nondegenerate critical/singular set}\label{C1}
Let $M$ be an interval and $f:M\to M$ be
a piecewise $C^2$ map: 
By this we mean that there exists a finite set \( \mathcal C' \) such
that \( f \) is \( C^2 \) and monotone on each
connected component of \( M\setminus\mathcal C'\) and admits a
continuous extension to the boundary so that 
\( f(c) : = \lim_{x\to c^{\pm}} f(x)\)  exists.
We denote by \( \mathcal C\)  the set of all ``one-sided critical
points'' $c^+$ and $c^{-}$ and define corresponding one-sided
neighbourhoods
\[
\Delta(c^{+},\delta) = (c^{+},c^{+}+\delta) \quad\text{and}\quad
\Delta(c^{-},\delta) = (c^{-}-\delta,c^{-}),
\]
for each $\delta>0$. For simplicity, from now on we use $c$ to
represent 
the generic element of \( \mathcal C\) and write $\Delta$ for
$\cup_{c\in\cC} \Delta(c,\delta)$.  We assume that each 
\( c\in\mathcal C\) has a well-defined (one-sided) critical order
$\ell = \ell(c) > 0$ in the sense that 
\begin{equation}\label{eq:der0}
 |f(x)-f(c)|
\approx  d(x,c)^{\ell} \quad \text{ and } \quad 
 |Df(x)|\approx  d(x,c)^{\ell-1} \quad \text{ and } \quad 
 |D^2f(x)|\approx  d(x,c)^{\ell-2}
\end{equation}
for all $x$ in some $\Delta(c,\delta)$. Note that we say that \(
f\approx g \) if  the ratio \( f/g \) is bounded above and below
uniformly in the stated domain. 
If  \( \ell(c) <1 \) we say that \( c \) is a \emph{singular} point
as this
implies unbounded derivative near \( c \); if  \( 1< \ell(c) \) we
say that \( c \) is a \emph{critical} point
as this implies that the 
derivative tends to \( 0 \) near \( c \). 
We shall assume also that \( \ell(c) \neq 1 \) for every \( c \) as
this would be a degenerate case which is not hard to deal with
but would require having to introduce special notation and special
cases,  whereas the other cases can all be dealt with in
a unified formalism.

\begin{remark}
For future reference we point out that this immediately implies
\begin{equation}\label{derdist}
 \frac{|D^{2}f(x)|}{|Df(x)|} \approx \frac{1}{d(x)} 
\end{equation}
 for all \( x \),  where \( d(x) \) denotes the distance of the point
\( x \) to the critical/singular set \( \mathcal C \) (indeed this is
the actual property of which we will make use). 
 \end{remark}

\subsubsection{Uniform expansion 
outside the critical neighbourhood}
\label{C2}
We suppose that $f$ is ``uniformly expanding away from
the critical points'', meaning that the following two conditions are
satisfied: there exists a constant \( \kappa>0 \), independent of \(
\delta \), such that for every point \( x \) and every integer \(
n\geq 1 \) such that $d(f^{j}(x),\mathcal
C)>\delta$ for all $0\le j \le n-1$ and $d(f^{n}(x),\mathcal
C) \leq \delta$ we have 
\begin{equation}\label{kappa}
    |Df^{n}(x)| \geq \kappa
\end{equation}
and,  for every $\delta>0$ there exist constants $c(\delta)>0$
and $\lambda(\delta)>0$ such that 
\begin{equation}\label{eq:ue}
|Df^{n}(x)| \ge c(\delta) e^{\lambda(\delta) n}
\end{equation}
for every $x$ and $n\ge 1$ such that $d(f^{j}(x),\mathcal
C)>\delta$ for all $0\le j \le n-1$.

We remark that both these conditions 
are quite natural and are often satisfied for smooth maps without
discontinuities.  More specifically, the first  one is satisfied 
if \( f \) is \( C^{3} \), has negative Schwarzian derivative and 
satisfies the property that the the derivative along all critical
orbits tends to infinity, see Theorem 1.3 in \cite{BruStr03}. 
The second  is satisfied in even greater generality, namely 
when $f$ is  $C^{2}$ and all periodic points are
repelling  \cite{Man85}. 

\subsubsection{Summability condition along the critical orbit}
\label{C3}
For each \( c\in\mathcal C \) we  write
\[
D_n(c)=|(f^n)'(f(c))|  \quad\text{ and } \quad d(c_{n}) =
d(c_{n},\mathcal C)
\]
to denote the derivative along the orbit of \( c \) and the distance
of \( c \) from the critical set respectively. 
We then assume that
for every critical point \( c \) with \( \ell = \ell(c) > 1 \) 
we have 
\begin{equation}\tag{\( \star \)
}
\sum_{n}
\frac{-n\log d(c_{n})}
{d(c_{n}) D_{n-1}^{1/(2\ell-1)}} < \infty.
 \end{equation}

\begin{remark}
This condition plays off the derivative against the recurrence in
such a way as to  optimize to some extent the class of maps to which
it applies. As mentioned in Section \ref{sumcon} above, we cannot
expect to obtain the conclusions of our main theorem in this setting
using a condition which only takes into account the growth of the
derivative. Notice that condition (\( \star \)) is satisfied 
 if the derivative is growing exponentially fast  and the recurrence
is not faster than 
exponential in the sense that 
\[ 
D_{n-1}\gtrsim e^{\lambda n}\quad  \text{ and } \quad  d(c_{n})
\gtrsim e^{-\alpha n} \quad 
\text{ with }  \quad  \alpha < \frac{\lambda}{2\ell-1}.
 \]
 Here and in the rest of the paper, 
 the symbol \( \gtrsim \)   means that the inequality holds up to 
 some multiplicative constant, i.e. there exists a constant \( C>0 \) 
 independent of \( n \) or any other constants, 
 such that  \( D_{n-1}\geq  e^{\lambda n} \)
 and \( d(c_{n}) \geq C e^{-\alpha n}  \). 
\end{remark}

\section{The main technical theorem}

\subsection{Inducing}
Our strategy for the proof is to  construct a countable 
partition \( \mathcal I \) of \( M \) (mod 0) into open intervals , define 
an inducing time function \( \tau: M \to \mathbb N \) 
which is constant on
elements of \( \mathcal I \), and  let 
 \( \hat f: M \to M \) denote the induced map defined by 
\[ \hat f(x) = f^{\tau(x)}(x). \]
This induced map is uniformly expanding on each element of \( \mathcal
I\) but does \emph{not} have many
desirable properties such as uniformly bounded distortion or long
branches. Nevertheless it has the two key properties we shall require
which are 
summable inducing times and summable variation. We recall that   the
\emph{variation} of a function
  $\varphi:M\to\RR$ over a subinterval $I=[a,b]$ of $M$ is
  defined by
$$
\var_I \varphi = \sup \sum_{i=1}^{N}
|\varphi(c_{i})-\varphi(c_{i-1})|
$$
where the supremum is taken over all $N\ge 1$ and all choices of
points
$a=c_{0}<c_{1}<\cdots<c_{N-1}<c_N=b$. 
For each \( I\in \mathcal I \) we define the function \( \omega_{I}:
M \to M \)
by 
\[ \omega_{I}(x)  = \left|\frac{\mathbb{1}_{I}(x)}{f'(x)}\right|
\]
Our main technical result in this paper is the following

\begin{theorem}
There exists a countable partition \( \mathcal I \) of \( M \) (mod 0)
and an inducing time function \( \tau: M \to \mathbb N \), 
constant on elements of \( \mathcal I \), such that the induced map
\( \hat f = f^{\tau(x)}(x) \) is uniformly expanding and satisfies
the 
 the following properties. 
\begin{enumerate}
\item (Summable variation) \[ \sum_{I\in\mathcal I}
\var_{M}\omega_{I} < \infty \]
\item (Summable inducing times) 
\[ \sum_{I\in \mathcal I} \tau(I) |I|< \infty \]
\end{enumerate}
\end{theorem}

Theorem 1 implies the Main Theorem by known arguments. Indeed, by a
result of Rychlik the summable variation property together with
uniform expansion implies that \(
\hat f \) admits a finite number of 
ergodic absolutely continuous invariant measure
whose basins cover \( I \) up to a set of measure zero
\cite{Ryc83,  Via97, BoyGor97}. By standard arguments the
summable inducing time property implies that these measures 
can be
pulled back to a absolutely continuous invariant probability measure
for the original map \( f \) satisfying the same properties 
\cite{MelStr93}. 

\begin{remark}
The arguments used in \cite{BruLuzStr03, AlvLuzPin04, DiaHolLuz06} 
also involve the
construction of an induced map with summable return times, but in
those papers the induced map has some very strong properties such as
uniformly bounded distortion and the Gibbs-Markov property (the
image of each partition element maps diffeomorphically to the entire
domain of definition of the induced map). To achieve these properties a
quite complicated construction is required, involving the inductive
definition of an infinite number of finer and finer partitions
together with a combinatorial and probabilistic argument showing that 
the procedure eventually converges. Besides the fact that we deal here
with a significantly larger class of systems, a major difference is
the construction of an induced map satisfying a different set of
conditions as formalized in the summable variation property stated in 
the theorem. These induced 
maps do not necessarily have bounded distortion and
there is no uniform lower bound for the size of the images. 
For this reason the construction of these induced maps is *much*
simpler, and in fact can be fully achieved in less than two pages of
text in the following section. The rest of the paper is just devoted
to checking the required properties.

\end{remark}

\subsection{Definition of the induced map}

The induced map \( \hat f \) can in fact be defined in complete
generality with essentially no assumptions on the map \( f \). We
will only require our assumptions to show that this induced map has
the desired properties. 
\subsubsection{Notation}
 For a point \( x \) in the neighbourhood 
\( \Delta(c,\delta)  \) of one of the critical points \( c \), we let 
\[ \hat I=\hat I_{0} = (x,c) \quad \text{ and } \quad \hat I_{j} =
(x_{j}, c_{j}) = (f^{j}(x), f^{j}(c)). \]
For an arbitrary interval \( I \) we let \( |I| \) denote the length
of \( I \) and \( d(I) \) denote it's distance to the critical set \(
\mathcal C \), i.e. the minimum distance of all point in \( I \) to
\( \mathcal C \). 
For each critical point \( c \) with  \( \ell=\ell(c) > 1 \), 
 and 
every integer \( n \geq 1 \) we let 
\begin{equation}\label{gamman}
\gamma_{n}(c) = \min\left\{\frac{1}{2}, \  
\frac{1}{d(c_{n}) D_{n-1}^{1/(2\ell-1)}}\right\}
\end{equation}
It  follows immediately  from 
the summability condition (\( \star \)) 
that 
\[ \sum_{n}\gamma_{n} < \infty.
\]

\subsubsection{Binding}
Given $c\in\mathcal{C}$,we define the \emph{binding period} of a
point 
$x\in \Delta(c,\delta)$ as follows. If \( \ell(c) < 1 \) we just
define the  binding period as \( p=1 \). Otherwise we define the
binding period as the \emph{smallest} $p=p(x)\in\NN$ such
that 
\[ 
|\hat I_{j}| \leq \gamma_j\,d(c_{j})  \text{ for } 1\leq j \leq p-1
\quad\text{ and } \quad
|\hat I_{p}| > \gamma_p\,d(c_{p}). 
\]
For each $c\in\cC$ and $p\ge 1$, define $I(c,p)$ to be the
interval of points $x\in\Delta(c,\delta)$ such that
$p(x)=p$. 
Observe that from the definition of binding it follows immediately
that 
\[
h(\delta):=\inf\{p(x):
  x\in\Delta(c,\delta),\ c\in\mathcal C\} \to \infty
  \]
\emph{monotonically } when $\delta\to 0$.  
Notice also that the interval $I(c,p)$ 
may be empty and indeed
that is the case, for instance, for all $p<h(\delta)$.

\subsubsection{Fixing \( \delta \)}\label{fixdelta}
Using the monotonicity of \( h(\delta) \) 
we can fix at this moment and for the rest of the paper 
\( \delta \) sufficiently small so that 
\begin{enumerate}
    \item the critical neighbourhood of size \( \delta \) of all
    critical/singular points are disjoint and the images of the
    critical/singular neighbourhoods are also disjoint from the
    critical/singular neighbourhoods themselves;
\item \( \gamma_{n}< 1/2 \) for all \( n\geq h(\delta) \);
\item \( D_{n-1}^{\frac{1}{2\ell-1}} \gg 2/\kappa \) for all \( n\geq
h(\delta) \). The symbol \( \gg \) here means that 
\( D_{n-1}^{\frac{1}{2\ell-1}}  \) must be larger than some constant
factor of \( 2/\kappa  \) for a constant which depends only on the map
itself and which is determined in the course of the proof but which
could in principle we specified explicitly at this point. 
\end{enumerate}

\subsubsection{Fixing \( q_{0} \)}
We now fix an integer \( q_{0}= q_{0}(\delta) \geq 1 \) sufficiently
large so that 
\[ 
C(\delta) e^{\lambda(\delta)q_{0}(\delta)} \geq 2. 
\]
Notice that the constants \( C(\delta) \) and \( \lambda(\delta) \)
come from the expansion outside the critical neighbourhoods given in
Section \ref{C2}. The choice of \( q_{0} \) is motivated by the fact
that any finite piece of orbit longer than \( q_{0} \) iterations
staying outside a \( \delta \) neighbourhood of the critical points
has an accumulated derivative of at least 2.

\subsubsection{The inducing time}

Let 
\[ 
M_{f}=\{x\in M: f^{i}(x) \notin \Delta \text{ for all } 0\le i <
q_{0}\}
\quad\text{ and } 
\quad M_{b} = M\setminus M_{f}
\]
so that \( M_{f} \)
denotes the set of points of \( M \) which remain outside \(
\Delta \) for the first \( q_{0}-1 \) iterations, and 
\( M_{b} \)
denotes those which enter \( \Delta \) at some time before \( q_{0}
\).
For \( x\in M_{b} \) let
\[ 
l_{0}=l_{0}(x) = \min\{0\leq l < q_{0}: f^{l}(x) \in \Delta\}
\quad\text{ and } \quad p_{0}=p_{0}(f^{l_{0}}(x))
\]
so that \( l_{0} \)
is the first time the orbit of \( x \) enters \( \Delta \) and \(
p_{0} \)
denotes the binding period corresponding to the point \( f^{l_{0}}(x)
\).
Then we define the inducing time by 
\begin{align}
\label{eq:indu2}
\tau (x) = 
\begin{cases} q_{0}&\text{ if } x\in M_{f}\\
l_{0}+p_{0} & \text{ if } x\in M_{b}.
\end{cases}
\end{align}

\subsubsection{The induced map}
We define the induced map as 
\[ \hat f(x) = f^{\tau(x)}(x) \]
and  let $\cI$ denote the partition of $M$ into the maximal
intervals restricted to which the induced map $\hat f$ is
smooth, and write $\cI_f = \cI | M_{f}$ and $\cI_b = \cI |
M_{b}$. 
This completes the definitions of the induced map.

\section{Variation, Distortion and Expansion}

In this section we prove a  
general formula relating the variation, the distortion and 
the expansion.  First of all we define the notion of
\emph{generalized distortion}. This is a very natural notion which is
no more difficult to compute than standard distortion and 
which appears in variation calculations.  Strangely it does not seem
to us to have been defined before in the literature.
For any interval \( I \) and
integer \( n\geq 1 \) we let \( I_{j}=f^{j}(I) \) for \( j=0,..., n
\) and 
define the \emph{(generalized) distortion} 
\[ \mathcal D(f^{n}, I) = 
 \prod_{j=0}^{n-1}\sup_{x_{j}, y_{j}\in I_{j}} 
\frac{|Df(x_{j})|}{|Df(y_{j})|}. 
\]
We remark here that we are taking the supremum over all choices of
sequences \( x_{j}, y_{j} \in I_{j} \). If these sequences are chosen
so that \( x_{j}=f^{j}(x), y_{j}=f^{j}(y) \) for some \( x,y\in I \)
then we recover the more standard notion of distortion. 
In particular, by choosing the sequence \( x_{j} \) arbitrary and the
sequence \(y_{j}=f^{j}(y) \) as the actual orbit of a point, we can
compare the
two products and,in this case,  the definition given above of
generalized
distortion immediately implies 
\begin{equation}\label{supbound}
\prod_{j=0}^{n-1}\sup_{I_{j}}\frac{1}{|Df|} \leq \frac{\mathcal
D(f^{n}, I)
}{|Df^{n}(x)|}
\end{equation}
for any \( x\in I \).  
For future reference we remark also that
by the mean value theorem, there exists some 
\( \xi_{j}\in I_{j} \) such that 
\[ \frac{Df(x_{j})}{Df(y_{j})} = 1+\frac{Df(x_{j})- Df(y_{j})
}{Df(y_{j})} 
= 1 + \frac{D^{2}f(\xi_{j})}{Df(y_{j})}|x_{j}-y_{j}|. \]
Therefore we have
  \begin{equation}\label{gendist}
\sup_{x_{j}, y_{j}\in I_{j}} 
\frac{|Df(x_{j})|}{|Df(y_{j})|}
\leq   1+ \frac{\sup_{I_{j}}|D^{2}f|}{\inf_{I_{j}} |Df|}
    |I_{j}|
\quad
\text{ and } 
\quad
\mathcal D(f^{n}, I) \leq \prod_{j=0}^{n-1}\left( 
1+ \frac{\sup_{I_{j}}|D^{2}f|}{\inf_{I_{j}} |Df|}
    |I_{j}| \right).
 \end{equation}

We are now ready to state the main result of this section.
\begin{lemma}
 \label{le:varest}   
    For any interval \( I  \) and integer \( l\geq 1 \) such that \( 
    f^{l}: I \to f^{l}(I) \) is a diffeomorphism, we have
    \[
    \var_{I}\frac1{|Df^{l}|}\lesssim \frac{\mathcal D(f^{l},
I)}{\inf_{I}|Df^{l}|}
     \cdot  \sum_{j=0}^{l-1}\int_{I_{j}}\frac{dx}{d(x)}.
    \]
 \end{lemma}   

Before starting the proof we recall
a few elementary properties of functions
with bounded variation which will be used here 
 and later on.  Proofs can be found, for instance, in
\cite{Via97} or \cite{BoyGor97}.
 For any interval $I \subset M$, $a, b\in\RR$, and 
 $\varphi, \psi:M\to\RR$,
 \begin{enumerate}
 \item[(V1)] $\var_{I} |\varphi| \le \var_{I} \varphi$;
 \item[(V2)] $\var_{I} (a\varphi+b\psi)
 \le |a|\var_{I}\varphi + |b|\var_{I}\psi$;
 \item[(V3)] $\var_{I} (\varphi\psi) 
 \le \sup_{I}|\varphi | \, \var_{I}\psi + 
 \var_{I}|\varphi| \, \sup_{I}\psi$;
 \item[(V4)] $\var_{J} \varphi = \var_{I} (\varphi\circ h)$
 if $h:I\to J$ is a homeomorphism;
 \item[(V5)] if $\varphi$ is of class $C^1$ then 
 $\var_{I}\varphi = \int_{I} |D\varphi(x)|\,dx$.
\item[(V6)] for any interval $I$, any bounded variation function
$\varphi$,
and any probability $\nu$ on $I$,
\begin{equation}\label{eq:invarsup}
\int_{I}\varphi\,d\nu - \var_{I}\varphi
\le \inf_{I}\varphi
\le \sup_{I}\varphi
\le \int_{I}\varphi\,d\nu + \var_{I}\varphi.
\end{equation}
In particular, this holds when $\nu =$ normalized Lebesgue 
measure on $I$.
 \end{enumerate}   
   
\begin{proof}
We start by writing
\[ \var_{I}\frac{1}{Df^{l}} = 
\var_{I}\left[\prod_{j=0}^{l-1} \frac{1}{Df}\circ f^{j} \right]
= \var_{I}\left[\left(\frac{1}{Df}\circ f^{l-1} \right)
\left(\prod_{j=0}^{l-2} \frac{1}{Df}\circ f^{j}
\right)\right] 
\]    
Thus, from (V3)  we have 
\[
\var_{I}\frac{1}{Df^{l}} \leq 
\left(\sup_{I} \frac{1}{|Df|}\circ f^{l-1} \right)
\left(\var_{I} \prod_{j=0}^{l-2} \frac{1}{Df}\circ f^{j}\right)+ 
\left(\var_{I}\frac{1}{Df}\circ f^{l-1}\right) 
\left(\sup_{I}\prod_{j=0}^{l-2} \frac{1}{|Df|}\circ f^{j}\right)
\]
Since the supremum of the product is clearly less than or equal to
the product of the supremums this gives 
\[
\var_{I}\frac{1}{Df^{l}} \leq 
\left(\sup_{I} \frac{1}{|Df|}\circ f^{l-1} \right)
\left(\var_{I} \prod_{j=0}^{l-2} \frac{1}{Df}\circ f^{j}\right)+ 
\left(\var_{I}\frac{1}{Df}\circ f^{l-1}\right) 
\left(\prod_{j=0}^{l-2} \sup_{I}\frac{1}{|Df|}\circ f^{j}\right)
\]
Thus, multiplying and dividing through by both the first and last
term of the right hand side of this expression, we get 
\begin{equation}\label{varl}
\var_{I}\frac{1}{Df^{l}}
\leq \left(\prod_{j=0}^{l-1}\sup_{I} \frac{1}{|Df(f^{j})|}\right)
\left[ 
\frac{\var_{I} \prod_{j=0}^{l-2} 
\frac{1}{Df(f^{j})}}{\prod_{j=0}^{l-2} \sup_{I} \frac{1}{|Df(f^{j})|}}
+\frac{\var_{I}\frac{1}{Df(f^{l-1})} }{\sup_{I}
\frac{1}{|Df(f^{l-1})|} } 
\right]
\end{equation}
We have used here the simplified notation \( [Df(f^{j})]^{-1} \) to
denote \( [Df]^{-1}\circ f^{j} \). Using this  bound recursively we
get 
\begin{equation}\label{varlminus1}
\var_{I} \prod_{j=0}^{l-2} 
\frac{1}{Df(f^{j})}
\leq \left(\prod_{j=0}^{l-2}\sup_{I}\frac{1}{|Df(f^{j})|}\right)
\left[ 
\frac{\var_{I} \prod_{j=0}^{l-3} 
\frac{1}{Df(f^{j})}}{\prod_{j=0}^{l-3} \sup_{I}\frac{1}{|Df(f^{j})|}}
+\frac{\var_{I}\frac{1}{Df(f^{l-2})} }{\sup_{I}
\frac{1}{|Df(f^{l-2})|} } 
\right]
 \end{equation}
 and therefore, substituting \eqref{varlminus1} into \eqref{varl} we
get 
 \[ 
 \var_{I}\frac{1}{Df^{l}}
\leq \left(\prod_{j=0}^{l-1} \sup_{I}\frac{1}{|Df(f^{j})|}\right)
\left[ 
\frac{\var_{I} \prod_{j=0}^{l-3} 
\frac{1}{Df(f^{j})}}{\prod_{j=0}^{l-3} \sup_{I}\frac{1}{|Df(f^{j})|}}
+ 
\frac{\var_{I}\frac{1}{Df(f^{l-2})} }{\sup_{I}
\frac{1}{|Df(f^{l-2})|} } 
+\frac{\var_{I}\frac{1}{Df(f^{l-1})} }{\sup_{I}
\frac{1}{|Df(f^{l-1})|} } 
\right]
  \]
  Continuing in this way and and then using (V4) we arrive at 
  \[ 
   \var_{I}\frac{1}{Df^{l}}
\leq \left(\prod_{j=0}^{l-1} \sup_{I} \frac{1}{|Df(f^{j})|}\right)
\left[ 
\sum_{j=0}^{l-1}
\frac{\var_{I}\frac{1}{Df(f^{j})} }{\sup_{I} \frac{1}{|Df(f^{j})|} } 
\right] = 
\left(\prod_{j=0}^{l-1} \sup_{I_{j}} \frac{1}{|Df|}\right)
\left[ 
\sum_{j=0}^{l-1}
\frac{\var_{I_{j}}\frac{1}{Df} }{\sup_{I_{J}} \frac{1}{|Df|} } 
\right] 
   \]
From the definition of generalized distortion, in particular
\eqref{supbound}, 
this gives 
\[ 
   \var_{I}\frac{1}{Df^{l}}
\leq 
\left(\prod_{j=0}^{l-1} \sup_{I_{j}} \frac{1}{|Df|}\right)
\left[ 
\sum_{j=0}^{l-1}
\frac{\var_{I_{j}}\frac{1}{Df} }{\sup_{I_{J}} \frac{1}{|Df|} } 
\right] 
\leq \frac{\mathcal D(f^{l}, I)}{\inf_{I}|Df^{l}|} 
\left[ 
\sum_{j=0}^{l-1}
\frac{\var_{I_{j}}\frac{1}{Df} }{\sup_{I_{j}} \frac{1}{|Df|} } 
\right].
 \]
Finally from  (V5)  and \eqref{derdist} 
we get 
\[ 
\var_{I_{j}}\frac{1}{Df} =
\int_{I_{j}}\left|\frac{D^{2}f}{(Df)^{2}}\right| 
\leq
\sup_{I_{j}}\frac{1}{|Df|}\int_{I_{j}}\left|\frac{D^{2}f}{Df}\right|dx
\lesssim \sup_{I_{j}}\frac{1}{|Df|}\int_{I_{j}} \frac{dx}{d(x,
\mathcal 
C)}.
\]
\end{proof}

\section{Binding}
\label{s:binding}

\subsection{Distortion during binding periods}

\begin{lemma}\label{le:distorcao}
For any \( x\in \Delta \), \( c\in \mathcal C \), the critical point
closest to \( x \),  \( \hat I_{0} = (x,c) \), and 
any \( 1\leq j\leq p(x)-1  \) we have 
\begin{equation}\label{bindingratio}
  \frac{|\hat I_{j}|}{d(\hat I_{j})} \leq  2 \gamma_{j}
  \quad\text{ and }\quad 
     \sup_{x_{j}, y_{j}\in\hat
I_{j}}\frac{|D^{2}f(x_{j})|}{|Df(y_{j})|} \lesssim \frac{1}{d(\hat
I_{j})}
\end{equation}
In particular there exists \( \Gamma>0 \) independent of \( x \)
such that  for all $1\le k \le p(x)-1$ we have
  \[ \mathcal D(f^{k}, \hat I_{1}) \leq \Gamma 
\quad\text{ and } 
\quad \int_{\hat I_{j}} \frac{1}{d(x)} dx \leq 2 \gamma_{j}  
\]
and for all \( y, z\in [x,c] \) we have
\[
 |Df^k(f(y))| \approx  |Df^k(f(z))|.
 \]
\end{lemma}

\begin{proof}
 The definition of binding period is designed to guarantee that the
 length \( |\hat I_{j}| \) of the interval 
 \(\hat  I_{j}=(f^{j}(x), f^{j}(c)) \) is small
 compared to its distance \( d(\hat I_{j}) \) to the critical set.
Indeed, 
 from the definition we have 
 \( d(\hat I_{j}) \geq d(f^{j}(c), \mathcal C) -
  d(f^{j}(c), f^{j}(x)) \geq (1-\gamma_{j}) d(f^{j}(c),
  \mathcal C) \)
  and therefore, for every \( 1\leq j\leq p-1  \) we have 
\[
  \frac{|\hat I_{j}|}{d(\hat I_{j})} \leq \frac{d(f^{j}(x),
f^{j}(c))}{(1-\gamma_{j})
  d(f^{j}(c), \mathcal C)}
  \leq \frac{\gamma_{j}}{1-\gamma_{j}}\leq 2 \gamma_{j}.
\]
In particular this also implies, from the order of the critical
points, that 
  \( \sup_{\hat I_{j}}|Df^{2}| \lesssim d(\hat I_{j})^{\ell-2}\) and 
\(  \inf_{\hat I_{j}}|Df| \gtrsim d(\hat
I_{j})^{\ell-1} \)
  and therefore 
\[
   \sup_{x_{j}, y_{j}\in \hat I_{j}}\frac{|D^{2}f(x_{j})|}{|Df(y_{j})|} 
   = 
\frac{\sup_{\hat I_{j}} |D^{2}f|}{\inf_{\hat I_{j}}
|Df|}   
\lesssim \frac{1}{d(\hat I_{j})}	 
\]
where \( \lesssim \) means that the bound holds up to a
multiplicative constant independent of \( \delta, I \) or \( j \).
  Now,  from \eqref{gendist} and \eqref{bindingratio}
   we have 
   \[ 
   \mathcal D(f^{k}, \hat I_{1})  \leq 
 \prod_{j=1}^{k}
\left(1+ \sup_{x_{j}, y_{j}\in I_{j}}\frac{|D^{2}f(x_{j})|}{
|Df(y_{j})|}|\hat I_{j}|\right)
\leq   \prod_{j=1}^{k} \left(1+  C\frac{|\hat  I_{j}|}{d(\hat I_{j})}\right)
\leq \prod_{j=1}^{k} (1+  2C \gamma_{j})
    \]
The right hand side is uniformly bounded by the summability of the \(
\gamma_{j} \)'s. Indeed, taking logs and using the inequality \( \log
(1+x) \leq x \) for all \( x\geq 0 \) we get \( \log \prod
(1+C\gamma_{j}) = \sum \log (1+C\gamma_{j})  \leq \sum C\gamma_{j}.
\)  This proves the uniform bound on the distortion \( \mathcal
D(f^{k}, \hat I_{1}) \). The fact that 
\(  |Df^k(f(x))| \approx  |Df^k(f(c))| \) then follows directly from
the definition of \( \mathcal D(f^{k}, \hat I_{1}) \) and the fact
that it is uniformly bounded. Finally notice that \( \int_{\hat
I_{j}}1/d(x) \leq |\hat I_{j}|/d(\hat I_{j}) \) and therefore the
required bound follows from \eqref{bindingratio}.
  \end{proof}

\subsection{The binding period partition}

The partition \( \mathcal I \) is defined quite abstractly and we do
not have direct information about the sizes of the partition elements 
and in particular the relation between their sizes and their distances
to the critical set. However, using the distortion bounds obtained
above, we can prove the following

\begin{lemma}\label{distreturn}
Let \( I \in\mathcal I \) with \( p(I) = p \) and \( I \) in the
neighbourhood of a critical point with order \( \ell \). 
Then 
\begin{equation}\label{distreturn1}
D_{p-1}^{-2/(2\ell-1)}  
\lesssim  \inf_{x\in I} d(x) \leq 
\sup_{x\in I}d(x) \lesssim D_{p-2}^{-2/(2\ell-1)}. 
 \end{equation}
In particular, letting \( \ell_{k}=\ell_{k}(c) \) denote the order of 
the critical/singular point closest to \( c_{k} \) we have 
\begin{equation}\label{distreturn2}
\mathcal D(f, I) \lesssim   \left[\frac{D_{p-1}}{D_{p-2}}\right]
  ^{\frac{ 2(\ell - 1)}{2\ell-1}} 
  \lesssim d(c_{p-1})^{ \frac{2(\ell - 1)(\ell_{p-1}-1)}{2\ell-1} }
\end{equation}
and
\begin{equation}\label{distreturn3}
\int_{I}\frac{1}{d(x)}dx \lesssim 
\log 
\left[\frac{D_{p-1}}{D_{p-2}}\right]
 ^{\frac{ 2(\ell - 1)}{2\ell-1}} 
 \lesssim \log 
 d(c_{p-1})^{-1}.
 \end{equation}
\end{lemma}

\begin{remark}
We remark that the distortion not uniformly bounded in \( p \) implying
that the induced map does not have uniformly bounded distortion.
Notice also that for some values of \( p \) it may happen that \(
D_{p-2}^{-2/(2\ell-1)} \ll  D_{p-1}^{-2/(2\ell-1)} \); in this case
the corresponding interval \( I \) would necessarily be empty, i.e.
there is no \( x \) with binding period \( p \).
\end{remark}

\begin{proof}
From Lemma \ref{le:distorcao}
 and the definition of binding period 
we have, for any \( x\in I \), 
\[ d(x) = |\hat I_{0}| \approx  |\hat I_{1}|^{1/\ell} 
\approx [D_{p-1}^{-1}|\hat I_{p}|]^{1/\ell} \geq
[D_{p-1}^{-1}\gamma_{p}d(c_{p})]^{1/\ell}
\]
and 
\[ 
d(x) = |\hat I_{0}| \approx  | \hat I_{1}|^{1/\ell} \approx 
[D_{p-2}^{-1} |\hat I_{p-1}| ]^{1/\ell}
\leq 
[D_{p-2}^{-1}\gamma_{p-1} d(c_{p-1})]^{1/\ell}
\]
By taking a sufficiently small \( \delta \) we can assume that \( p \)
is sufficiently large so that \( \gamma_{p-1}, \gamma_{p} < 1/2\) and 
therefore, from the definition of the sequence \( \{\gamma_{n}\} \) we
get 
\[ \gamma_{n}d(c_{n}) = D_{n-1}^{-1/(2\ell-1)}.
\]
Thus, substituting into the expressions above gives 
\[ d(x) \gtrsim 
[D_{p-1}^{-1}\gamma_{p}d(c_{p})]^{1/\ell} 
= [D_{p-1} D_{p-1}^{-1/(2\ell-1)}]^{1\ell} =
[D_{p-1}^{-2\ell/(2\ell-1)}]^{1/\ell} = D_{p-1}^{-2/(2\ell-1)}
\]
and, similarly,  
\[ 
d(x) \lesssim D_{p-2}^{-2/(2\ell-1)}.
\]
This gives the first set of inequalities. 
As a consequence we immediately get
\[ 
D_{p-2}^{-2(\ell-1)/(2\ell-1)}
\gtrsim \sup_{I}|Df(x)| \geq \inf_{I} |Df(x)| \gtrsim
D_{p-1}^{-2(\ell-1)/(2\ell-1)}
 \]
 and therefore, 
 \[ 
 \mathcal D(f, I) = \sup_{x,y \in I}\frac{|Df(x)|}{|Df(y)|} \lesssim 
 \left[\frac{D_{p-1}}{D_{p-2}}\right]
  ^{\frac{ 2(\ell - 1)}{2\ell-1}} 
    \]
  This gives the first inequality in \eqref{distreturn2}. To get the
  second inequality we simply use the fact that \( D_{p-1}\approx
  D_{p-2}d(c_{p-1})^{\ell_{p-1}-1} \). 
 To get the last inequality we simply integrate \( 1/d(z) \) over the
interval \( I=(x,y) \)
to get 
  \[ 
  \int_{I}\frac{1}{d(z)}dx = |\log d(x) - \log d(y)| \lesssim \log 
   \left[\frac{D_{p-1}}{D_{p-2}}\right]^{2/(2\ell-1)}
   \]
   and then argue as above.
 \end{proof}   

\subsection{Expansion during binding periods}

\begin{lemma}\label{le:london1}
For all  $c\in\mathcal{C}$ , 
$x\in\Delta(c,\delta)$
and \( p=p(x) \), we have 
\begin{equation}\label{eq1:london1}
|Df^p(x)| \gtrsim  D_{p-1}^{\frac{1}{2\ell-1}} 
\end{equation}
In particular we can choose \( \delta \) small enough so that 
\[
|Df^p(x)|  \geq 2/\kappa.
\] 
\end{lemma}

\begin{proof}
Using the chain rule, bounded distortion in  binding periods
 and Lemma \ref{distreturn} we have 
\[ 
|Df^{p}(x)| = |Df^{p-1}(f(x)) \cdot Df(x)| 
\gtrsim D_{p-1}
D_{p-1}^{-2(\ell-1)/(2\ell-1)} 
= D_{p-1}^{\frac{1}{2\ell-1}} 
 \]
This gives \eqref{eq1:london1}. The inequality 
\( |Df^p(x)|  \geq 2/\kappa \) then just follows from the choice of \( 
\delta \) in Section \ref{fixdelta}. 
\end{proof}

\section{Inducing}

\subsection{Expansion of the induced map}

\begin{lemma}\label{le:supzero}
For every $x\in M$ we have
\[ 
|D\hat f(x)| \geq 2.
 \]

\end{lemma}    

\begin{proof}
This follows immediately from the definition of the induced map and
the expansion estimates during binding periods obtained in Lemma
\ref{le:london1} together with conditions \ref{kappa} and 
\ref{eq:ue}, the choice of
\( \delta \) and the corresponding choice of \( q_{0} \). 
\end{proof}

\subsection{Distortion of the induced map}
We now study the distortion of the induced map \( \hat f \) on each
of its branches.

\begin{lemma}\label{le:globaldist}
There exists a constant \( D=D(\delta)>0 \) such that 
\begin{equation}\label{globaldist0}
    \mathcal D(f^{\tau}, I) \leq D 
\quad\text{ and } \quad
\mathcal D(f^{\tau}, I) \lesssim 
d(c_{p-1})^{ \frac{2(\ell - 1)(\ell_{p-1}-1)}{2\ell-1} }
\end{equation}
for all \( I\in \mathcal M_{f} \) (in which case \( \tau \equiv q_{0}
\))
and \( I\in \mathcal M_{b} \) (in which case \( \tau = l+p \))
respectively, where \( \ell \) is the order of the critical point
associated to \( I_{l} \).
Also, we have 
\[ \sum_{j=0}^{\tau-1}
\int_{I_{j}} \frac{1}{d(x)}dx \leq  D
\]
and
\begin{equation}\label{globaldist1}
\sum_{j=0}^{\tau-1}
\int_{I_{j}}\frac{1}{d(x)}dx \leq D +
\log d(c_{p-1})^{-1}
\end{equation}
respectively for \( I\in \mathcal M_{f} \) and \( I\in \mathcal M_{b}
\).
\end{lemma}
\begin{proof}
For \( I\in \mathcal I_{f} \) we have standard distortion estimates
for uniformly expanding maps which give a uniform distortion bound
\( D \) depending on the size of \( \Delta \). For \( I\in \mathcal I_{b} \) on the other
hand we write 
\[ \mathcal D(f^{\tau}, I) = \mathcal D(f^{l}, I)\cdot \mathcal D(f,
I_{l}) \cdot \mathcal D(f^{p-1}, I_{l+1}).
 \]
The first term consists of  iterates for which \( I_{j} \) lies
always outside 
\( \Delta \) and therefore is bounded above by the same constant \( D
\) as above. The second and third term have already been estimated
above in Lemmas \ref{le:distorcao} and \ref{distreturn}. Combining
these estimates we complete the first set of estimates. 

For \( I\in \mathcal M_{f} \), 
using the uniform expansion outside \( \Delta \) we have \(
|I_{j}|\leq c(\delta)^{-1}e^{-\lambda(\delta)(\tau-j)} \) and
therefore 
\[ 
 \sum_{j=0}^{\tau-1} \int_{I_{j}}\frac{1}{d(x)}dx 
 \leq  \sum_{j=0}^{\tau-1} 
 \frac{|I_{j}|}{d(I_{j})} 
 \leq  \sum_{j=0}^{\tau-1} 
 \frac{c(\delta)^{-1}e^{-\lambda(\delta)(\tau-j)} }{\delta}
 \leq D.  
 \]
 For \( I\in \mathcal M_{b} \) we again split the sum into three
parts corresponding to the initial iterates outside \( \Delta \), the
first iterate in \( \Delta \), and the following binding period. The
fist part of the sum is bounded by the same constant \( D \) as
above. The second and third have already been estimated above. Thus,
from Lemmas \ref{le:distorcao} and \ref{distreturn} and in particular 
\eqref{distreturn3} we get the statement. 

\end{proof}

\section{Summability}

We are now ready to prove the summable variation and the summable
inducing time 
properties.

\subsection{Summable variation}

From the definition of \( \omega_{I} \) that we have 
\[ \var_{M} \omega_{I} = \var_{I} \omega_{I} + 2\sup_{I} \omega_{I} 
=  \var_{I}\frac{1}{|Df^{\tau}|} 
  + 2 \sup_{I}\frac{1}{|Df^{\tau}|}\]
  For the supremum we have, from Lemma \ref{le:supzero}, 
  \begin{equation}\label{sup}
  \sup_{I}\frac{1}{|Df^{\tau}|} \leq \frac{1}{
  D_{p-1}^{1/(2\ell-1)}}
   \end{equation}
and for the variation, we have, 
substituting the estimates in \eqref{sup}, \eqref{globaldist0} and
 \eqref{globaldist1} into 
the formula obtained in Lemma  \ref{le:varest}, 
\[
 \var_{I}\frac1{|Df^{\tau}|}\lesssim \frac{\mathcal D(f^{\tau},
I)}{\inf_{I}|Df^{\tau}|}
  \cdot  \sum_{j=0}^{\tau-1}\int_{I_{j}}\frac{dx}{d(x)}
  \lesssim 
 \frac{D+\log d(c_{p-1})^{-1}}
 {d(c_{p-1})^{ - \frac{2(\ell - 1)(\ell_{p-1}-1)}{2\ell-1} }
 D_{p-1}^{1/(2\ell-1)}}. 
  \]
We can write 
\[ D_{p-1}^{\frac{1}{2\ell-1}}
\approx D_{p-2}^{\frac{1}{2\ell-1}}
d(c_{p-1})^{\frac{\ell_{p-1}-1}{2\ell-1}}
\]
and 
\[ 
d(c_{p-1})^{ - \frac{2(\ell - 1)(\ell_{p-1}-1)}{2\ell-1} }
d(c_{p-1})^{\frac{\ell_{p-1}-1}{2\ell-1}} = 
d(c_{p-1})^{ - \frac{(2\ell - 1)(\ell_{p-1}-1)}{2\ell-1} }
= d(c_{p-1})^{(1-\ell_{p-1})}
\]
and so, substituting above, gives
\begin{equation}\label{gap}
 \var_{I}\frac1{|Df^{\tau}|} \lesssim 
\frac{D+\log d(c_{p-1})^{-1}}
{d(c_{p-1})^{(1-\ell_{p-1})} D_{p-2}^{1/(2\ell-1)} } \leq 
\frac{D+\log d(c_{p-1})^{-1}}
{d(c_{p-1}) D_{p-2}^{1/(2\ell-1)}}. 
\end{equation}

The summability then follows immediately from \( (\star) \).

  \subsection{Summable inducing times}
To prove the summability of the inducing time notice first of all
that 
the number of intervals of a given inducing time is uniformly
bounded. 
Therefore it is sufficient to prove the summability with respect to
the binding time. For this we give a basic upper bound for  the size
of each element \( I\in\mathcal I \) using the mean value theorem and
Lemma \ref{le:supzero}. This gives 
\[\sum \tau(I) |I| \lesssim \sum_{p} p |I| \lesssim
\sum_{p}\frac{p}{D_{p-1}^{1/(2\ell-1)}}. 
\]
Again, the summability follows directly from \( (\star) \).
This completes the proof of the Theorem.

\subsection{Final remarks}
    Notice that there is a significant gap between the first bound and 
    the second bound in \eqref{gap}, particularly
    in the special case in which there are no singularities and where 
    therefore \( \ell>1 \) for every critical point. In this case we
    get 
 \[ 
 \var_{I}\frac1{|Df^{\tau}|} \lesssim 
 \frac{\log d(c_{p-1})^{-1}}
 {d(c_{p-1})^{(1-\ell_{p-1})} D_{p-2}^{1/(2\ell-1)} } \leq 
 \frac{\log d(c_{p-1})^{-1}}
 {D_{p-2}^{1/(2\ell-1)}}. 
 \]   
This leaves only an extremely mild condition on the recurrence of the 
critical points and therefore the summability conditions 
reduces almost to the condition \(
\sum 1/D_{n}^{1/(2\ell-1)} \) assumed for smooth maps in
\cite{BruLuzStr03}. Ideally we would therefore like to replace
condition \( (\star) \) by the summability condition
\begin{equation*}\tag*{\( (\star\star) \)}
    \sum_{n}  \frac{n\log d(c_{n})^{-1}}
 {d(c_{n})^{(1-\ell_{n})} D_{n-1}^{1/(2\ell-1)}}< \infty
 \end{equation*}
which would automatically reduce to the condition
\[ \sum_{n}  \frac{n\log d(c_{n})^{-1}}
 {D_{n-1}^{1/(2\ell-1)}}< \infty
\]
in the smooth case. This however gives rise to technical difficulties 
that we have not been able to overcome, mainly in the definition of 
the sequence \( \gamma_{n} \), recall \eqref{gamman}. Condition \(
(\star\star) \) does not imply the summability of the \( \gamma_{n} \)
with the definition given in \eqref{gamman} and on the other hand,
changing the definition of \( \gamma_{n} \) to something more natural 
in terms of \( (\star\star) \), such as for example 
\( 1/(d(c_{n})^{(1-\ell_{n})} D_{n-1}^{1/(2\ell-1)}) \) 
gives rise to additional complications in the calculations and
estimates related to the binding period in Section \ref{s:binding}. It
is not clear to us whether these are superficial difficulties which
can be overcome or whether they reflect deeper issues.


\end{document}